\documentclass [11pt,twoside,a4paper]{article}
\usepackage{amsfonts}
\usepackage{amsthm}
\usepackage{amsmath}
\usepackage{amstext}
\usepackage{amssymb}
\usepackage{mathrsfs}
\usepackage{amscd}
\usepackage{xypic}
\usepackage{epsf}              
\usepackage{graphicx}          
\usepackage{fancybox}          
\usepackage{color}             
\usepackage{fancyhdr}
\usepackage[hang,footnotesize]{caption2}  

\def\mathcal{\mathscr}
\newfont{\aaa}{cmb10 at 19pt}
\newfont{\bbb}{cmb10 at 11pt}

\newcommand{\bm}[1]{\mbox{\boldmath $#1$}}

\def\v1{\vspace{1mm}}

\def\leq{\leqslant}

\def\geq{\geqslant}

\def\ZZ{\mathbb Z}

\newcommand{\beq}{\begin{equation}}
\newcommand{\eeq}{\end{equation}}
\newcommand{\bey}{\begin{eqnarray}}
\newcommand{\eey}{\end{eqnarray}}
\newcommand{\beqs}{\begin{eqnarray*}}
\newcommand{\eeqs}{\end{eqnarray*}}

\def\lm{{\textsf{Pra}(M)}}
\def\ln{{\textsf{Pra}(N)}}
\def\mc{{\mathbb C}}
\def\ad{{\hbox{\rm ad}}}
\def\cl{{\mathcal L}}
\def\bn{\underline{n}}
\def\bm{\underline{m}}
\def\bk{\underline{k}}
\def\br{\underline{r}}

\newtheorem{theo}{Theorem}

\newtheorem{defi}[theo]{Definition}
\newtheorem{lemma}[theo]{Lemma}

\newtheorem{remark}[theo]{Remark}

\makeatletter
\def\@evenhead{
\vbox{\hbox to \textwidth {}{\hspace{0mm}{\footnotesize
\thepage}}{\hspace{8cm} {\footnotesize {Li-meng Xia, Naihong Hu}}} \protect\vspace{1truemm}\relax \hrule depth0pt
height0.15truemm width\textwidth}}
\def\@evenfoot{}
\def\@oddhead{\vbox{\hbox to \textwidth
{{\hspace{0cm}{\footnotesize Lie algebras from intersection matrices}\hfill{\footnotesize
\thepage}}\hspace{0mm}}{} \protect\vspace{1truemm}\relax\hrule
depth0pt height0.15truemm width\textwidth}}
\def\@oddfoot{}
\makeatother

\begin{document}

\thispagestyle{empty} \thispagestyle{fancy} {
\fancyhead[RO,LE]{\scriptsize \bf 
} \fancyfoot[CE,CO]{}}
\renewcommand{\headrulewidth}{0pt}


\setcounter{page}{1}
\qquad\\[8mm]

\noindent{\aaa{A class of Lie algebras arising from intersection matrices}}\\[1mm]

\noindent{\bbb Li-meng Xia$^{1,\dag}$,\quad Naihong Hu$^{2}$}\\[-1mm]

\noindent\footnotesize{$^1$Faculty of Science, Jiangsu University, Zhenjiang, 212013, Jiangsu Prov. China,\\
$^2$Department of Mathematics, East China Normal University, Shanghai, 200241, China}\\[6mm]


\normalsize\noindent{\bbb Abstract}\quad In present work, we find a class of Lie algebras, which are defined from the symmetrizable generalized intersection matrices. However, such algebras are different from generalized intersection matrix algebras and intersection matrix algebras. Moreover, such Lie algebras generated by  semi-positive definite matrices can be classified by the modified Dynkin diagrams.
\vspace{0.3cm}

\footnotetext{
\hspace*{5.8mm}Corresponding author: Li-meng Xia, E-mail: xialimeng@ujs.edu.cn}

\noindent{\bbb Keywords}\quad intersection matrices; extended affine Lie algebras; generators; classification\\
{\bbb MSC}\quad 17B22, 17B65\\[0.4cm]

\noindent{\bbb 1\quad Introduction}\\[0.1cm]

In the early to mid-1980s, Peter Slodowy discovered that matrices like
$$M =\left[\begin{array}{cccc}
2&-1&0&1\\
-1&2&-1&1\\
0& -2&2&-2\\
1&1&-1&2
\end{array}\right]$$
were encoding the intersection form on the second homology group of
Milnor fibres for germs of holomorphic maps with an isolated singularity
at the origin \cite{S1}, \cite{S2}. These matrices were like the generalized
Cartan matrices of Kac-Moody theory in that they had integer entries,
$2$'s along the diagonal, and $m_{ij}$ was negative if and only if $m_{ji}$ was
negative. What was new, however, was the presence of positive entries
off the diagonal. Slodowy called such matrices generalized intersection
matrices:

\begin{defi}{\rm (\cite{S1})} An $n\times n$ integer-valued matrix $M=(m_{ij})_{1\leq i,j\leq n}$ is called a
generalized intersection matrix ({\bf\emph{gim}}) if

\qquad $m_{ii}=2$,

\qquad $m_{ij}<0$ if and only if $m_{ji}<0$, and

\qquad $m_{ij}>0$ if and only if $m_{ji}>0$

{\noindent for $1\leq i, j\leq n$ with $i\not= j$.}
\end{defi}

Slodowy used these matrices to define a class of Lie algebras that
encompassed all the Kac-Moody Lie algebras:
\begin{defi}[\cite{S1},\cite{BrM}]. Given an $n\times n$ generalized intersection
matrix $M = (m_{ij})$, define a Lie algebra over $\mc$, called a generalized
intersection matrix (${\textbf{gim}}$) algebra and denoted by ${\textbf{gim}}(M)$, with:

generators: $e_1, . . . , e_n, f_1, . . . , f_n, h_1, . . . h_n$,

relations:

(R1) for $1\leq i, j \leq n$,
\beqs [h_i, e_j ] = m_{ij}e_j,&{[h_i, f_j ]} =-m_{ij}f_j,&{[e_i, f_i]} = h_i,\eeqs

(R2) for $m_{ij}\leq 0$,
\beqs [e_i, f_j ] = 0 = [f_i, e_j ],&&(\ad e_i)^{-m_{ij}+1}e_j = 0 = (\ad f_i)^{-m_{ij}+1}f_j,\eeqs

(R3) for $m_{ij} > 0$, $i\not= j$,
\beqs [e_i, e_j ] = 0 = [f_i, f_j ],&&(\ad e_i)^{m_{ij}+1}f_j = 0 = (\ad f_i)^{m_{ij}+1}e_j.\eeqs
\end{defi}

If the $M$ that we begin with is a generalized Cartan matrix, then
the $3n$ generators and the first two groups of axioms, (R1) and (R2),
provide a presentation of the Kac-Moody Lie algebras \cite{C}, \cite{GbK}, \cite{K}.

Slodowy and, later, Berman showed that the $\emph{\textbf{gim}}$ algebras are also
isomorphic to fixed point subalgebras of involutions on larger Kac-
Moody algebras \cite{S1}, \cite{Br}. So, in their words, the $\emph{\textbf{gim}}(M)$ algebras lie
both "beyond and inside" Kac-Moody algebras.

Further progress came in the 1990s as a byproduct of the work of
Berman-Moody, Benkart-Zelmanov, and Neher on the classification
of root-graded Lie algebras \cite{BrM}, \cite{BnZ}, \cite{N}. Their work
revealed that some families of intersection matrix ($\emph{\textbf{im}}$) algebras, were universal covering algebras
of well understood Lie algebras.
An $\emph{\textbf{im}}$ algebra generally is a quotient algebra of a $\emph{\textbf{gim}}$ algebra associated to the ideal generated  by homogeneous vectors those have long roots (i.e., $(\alpha,\alpha)>2$).

A handful of other researchers also began
engaging these new algebras. For example, Eswara-Moody-Yokonuma used vertex operator representations to show that $\emph{\textbf{im}}$ algebras
were nontrivial \cite{EMY}. Gao examined compact forms of $\emph{\textbf{im}}$ algebras
arising from conjugations over the complex field \cite{G}.
Berman-Jurisich-Tan showed
that the presentation of $\emph{\textbf{gim}}$ algebras could be put into a broader
framework that incorporated Borcherds algebras \cite{BrJT}.

Peng found
relations between $\emph{\textbf{im}}$ algebras and the representations of tilted algebras
via Ringel-Hall algebras \cite{P}. Especially, Peng-Xu studied the root system of GIMs in \cite{PX} and defined a new class of Lie algebras in \cite{XP}.
The Peng-Xu algebra is invariant under the action of braid group, and it can be classified by the root system when  GIM is semi-positive definite.

In present paper, for a symmetrizable generalized intersection matrix $M$, a Lie algebra $\lm$ is defined. Our construction is motivated by the \emph{${\textbf{gim}}$} algebras, $\emph{\textbf{im}}$ algebras and the extended affine Lie algebras. Such an algebra is named here by \textsf{partial reflection algebra}.  For indecomposable symmetrizable generalized intersection matrices, the partial reflection algebras have properties:

\quad\quad $\diamond$  they are quotients of \emph{${\textbf{gim}}$} algebras and different from $\emph{\textbf{im}}$ algebras;

\quad\quad $\diamond$  they can be classified by modified Dynkin diagrams for semi-positive definite case;

\quad\quad $\diamond$  if $M$ is positive, then $\lm$ is finite simple;

\quad\quad $\diamond$  if $M$ has co-rank one, then $\lm$ is an affine Lie algebra.

If there exists a diagonal non-degenerate matrix $S={\rm diag}(s_1,\cdots,s_n)$ such that $SM$ is symmetric, then $M$ is called symmetrizable. In present work, $M$ is always assumed to be a symmetrizable generalized intersection matrix with rank $n-K$.

\begin{lemma}
For symmetrizable generalized intersection matrix $M$, there exists a $\mc$-space $V$ with  symmetric non-degenerate bilinear form $(-,-)$ satisfying:
\beqs \dim V&=&n+K,\eeqs
$V$ has a prime root system $\{v_1,\cdots,v_{n}, u_1,\cdots,u_K\}$ such that
\beqs \frac{2(v_i,v_j)}{(v_i,v_i)}&=&m_{ij},\eeqs
for $1\leq i,j\leq n$.
\end{lemma}

For any non-isotropic element $v$(i.e., $(v,v)\not=0$), there is a reflection
\bey \rho_{v}:&&\lambda\mapsto \lambda-\frac{2(\lambda,v)}{(v,v)}v.\eey

\begin{defi}
Suppose that $M=(m_{ij})_{n\times n}$ is a GIM and $\Pi=\{v_1,\cdots,v_n\}\subset V$ is a  prime root system such that $\frac{2(v_i,v_j)}{(v_i,v_i)}=m_{ij}$, where $V$ has a non-degenerate quadratic form $(\cdot, \cdot)$ and $\dim V=2n-rank(M)$. Let $\Pi'\subset V$.
We say that $\Pi$ and $\Pi'$ are braid-equivalent (denoted by $\Pi\sim\Pi'$) if there exists a sequence of transformations
of the form:
\beqs \Pi\mapsto\cdots \Pi_k\mapsto \Pi_{k+1}\cdots\Pi',\eeqs
where \beqs \Pi_k\mapsto \Pi_{k+1}=(\Pi_k\setminus\{\beta\})\cup\{\rho_\alpha(\beta)\}\eeqs
for some $\alpha, \beta\in\Pi_k$. Particularly, $\Pi'$ is called a braid-equivalent basis of $M$.
\end{defi}

\begin{defi}
Let $\Pi, \Pi'$ are prime root systems of GIMs $M$ and $N$, respectively. We say that $M$ and $N$ are braid-equivalent if $\Pi$ and $\Pi'$ is braid-equivalent.
\end{defi}

The referee reminded us to notice Peng-Xu's previous work. The above definitions are analogues of which appeared in \cite{PX}(also see \cite{XP})  and we also adopt their terminology {\bf braid-equivalent}. In fact, this equivalence relation corresponds to  the reflections of single lines(not of the whole space).
This is why we adopt the name \textsf{partial reflection algebra}.

The \textsf{partial reflection algebra} is dependent on the braid-equivalent basis. However, this Lie algebra is different from that defined  by Peng-Xu, the difference will be showed partially by Example 3) in below. If the GIM is semi-positive definite, an interesting thing is that both classes of Lie algebras defined by Peng-Xu and us have the same classification (see Theorem 11 below).

\noindent
\\[4mm]

\noindent{\bbb 2\quad Lie algebra $\lm$}\\[0.1cm]

Suppose that $M$ is a symmetrizable intersection matrix:
\beqs M=\left(m_{ij}\right)_{i,j\leq n},\eeqs
and ${\rm corank}M=n-{\rm rank}M=K$.

Define lattices  $P^\vee$ as:
\bey P^\vee&=&\bigoplus_{i=1}^n\ZZ h_i\oplus\bigoplus_{j=1}^{K}\ZZ d_j,\\
P&=&\{\lambda\in H^*\mid \lambda(P^\vee)\subseteq\ZZ\},\eey
where $H^*$ is the dual space of
\bey H&=&\mc\otimes_\ZZ P^\vee,\eey
and  $H^*$  has a subset
\bey \Pi&=&\{\alpha_i\in H^*\mid i=1,\cdots,n; \alpha_i(h_j)=m_{ji}\}\eey
which is linearly independent.

There exists a bilinear form over $H^*$, such that $2(\alpha_i,\alpha_j)/(\alpha_j,\alpha_j)=m_{ij}.$

\begin{defi}
A boundary reflection Lie algebra $\lm$ associated with $(M,P^\vee, P,\Pi)$ is the Lie algebra over complex number field $\mc$ generated by $e_i,f_i(i=1,\cdots,n), h\in H$ with defining relations:
\bey
[e_i,f_i]=h_i,&[h,e_i]=\alpha_i(h)e_i,&[h,f_i]=-\alpha_i(h)f_i,
\eey
associated to the graded decomposition
$$\lm=\sum_{\alpha}\lm_\alpha,\quad \lm_\alpha=\{v\mid [h,v]=\alpha(h)v,\forall h\in H\},$$
for all ${\mathcal B}\sim\Pi$ and all $\alpha,\beta\in {\mathcal B}$, $x\in \lm_\alpha, y\in \lm_\beta, z\in \lm_{-\alpha},$
\bey
(\ad x)^{1-\frac{2(\alpha,\beta)}{(\alpha,\alpha)}}(y)=0, &[z,y]=0,&(\alpha,\beta)<0,\\
(\ad z)^{1+\frac{2(\alpha,\beta)}{(\alpha,\alpha)}}(y)=0, &[x,y]=0,&(\alpha,\beta)>0,\\
{[x,y]}=0, &[z,y]=0,&(\alpha,\beta)=0.\eey\end{defi}

\begin{lemma}
If $M$ is a generalized Cartan matrix(GCM) of finite type or of affine type, then $\lm$ is a generalized Kac-Moody Lie algebra.
\end{lemma}
\begin{proof}
(1) If $M$ is of simply-laced type, then the result holds by the properties of root vectors of real roots.

(2) If $M$ is of order $2$, then this lemma holds by the definition of Kac-Moody algebra.

(3) If $M$ is of type $B_l$, $B_l^{(1)}$ or $G_2^{(1)}$, then any ${\mathcal B}$ contains one short root, the relations (2.6)-(2.8) hold by the properties of root vectors of real roots.

(4) Assume that $M$ is of type $C_l$. If $(\rho_\alpha(\beta),\gamma)=0$ and $\rho_\alpha(\beta)\pm\gamma$ are roots, then $\rho_\alpha(\beta)\pm\gamma$ are long roots.

There is a unique long root $\alpha^*\in{\mathcal B}$, then ${\mathcal B}\sim{\mathcal B}_{\omega, \alpha^*}$ for each $\omega\in W$, where $W$ is the Weyl group and
\beqs {\mathcal B}_{\omega, \alpha^*}=({\mathcal B}\setminus\{\alpha^*\})\cup\{\omega(\alpha^*)\}.\eeqs
Let $W_0:=W_{S,{\mathcal B}}$ be the subgroup generated by $\{\rho_\mu| \mu\in{\mathcal B}, (\mu,\mu)<(\alpha^*,\alpha^*)\}$.
If ${\mathcal B}\cap\{\pm(\rho_\alpha(\beta)\pm\gamma)\}$ is empty,  then $\rho_\alpha(\beta)+\gamma-\alpha^*$ (or $\rho_\alpha(\beta)+\gamma+\alpha^*$) is two times of a combination of short roots in ${\mathcal B}$. Then there exists $\omega\in W_0$  such that $\lambda:=\omega(\alpha^*)=\rho(\beta)+\gamma$ (or $\lambda:=\omega(\alpha^*)=-(\rho_\alpha(\beta)+\gamma)$) and hence \beqs \{\lambda, \rho_\lambda(\gamma), \rho_\alpha(\beta)\}\subset ({\mathcal B}_{\omega, \alpha^*})_{\rho_\lambda, \gamma}\sim {\mathcal B}_{\omega, \alpha^*}\sim{\mathcal B}.\eeqs
Then  $({\mathcal B}_{\omega, \alpha^*})_{\rho_\lambda, \gamma}$ contains two same roots, and we obtain a contradiction.
For the case $\lambda:=\omega(\alpha^*)=-(\rho_\alpha(\beta)+\gamma)$, we have
\beqs \{\lambda, \rho_\lambda(\gamma), -\rho_\alpha(\beta)\}\subset (({\mathcal B}_{\omega, \alpha^*})_{\rho_\lambda, \gamma})_{\lambda(\gamma), -\rho_\alpha(\beta)}\sim {\mathcal B}_{\omega, \alpha^*}\sim{\mathcal B},\eeqs
and the same contradiction is obtained.

(5) If $M$ is of type $F_4$, let $\alpha_1,\alpha_2$ be the long roots and $\beta_1, \beta_2$ be the short roots in ${\mathcal B}$. Then for each long root $\alpha^*$, there exists
\beqs  \alpha\in\{\pm\alpha_1, \pm\alpha_2, \pm(\alpha_1+\alpha_2)=\pm\rho_{\alpha_1}(\alpha_2)\}\eeqs
such that $\alpha^*$ belongs to the $W_0$-orbit of $\alpha$, where $W_0=\langle \rho_{\beta_1}, \rho_{\beta_2}\rangle$. The method of (4) works for this case.

(6) Assume $M$ is of type $C^{(1)}$, $F_4^{(1)}$ or $E_6^{(2)}, A_{2l-1}^{(2)}$. If $\alpha_1, \alpha_2$ are two long roots and $\alpha_1-\alpha_2$ is isotropic, then $\alpha_1-\alpha_2$ is an even multiples of the principal imaginary root. Similar to (4) and (5), we also obtain the result.

(7) If $M$ is of type $A_{2l}^{(2)}$, then each ${\mathcal B}$ contains a unique longest root and a unique shortest root. So the proof is similar to (4).

(8) If $M$ is of type  $D_{l+1}^{(2)}$, then  each ${\mathcal B}$ contains two short roots $\beta, \gamma$. If $(\beta, \gamma)=0$, then $\beta\pm\gamma$ are not roots. If $(\beta, \gamma)\not=0$, then $\langle\beta,\gamma\rangle=\pm2$ and $\beta\mp\gamma$ is an imaginary root. Hence the relations (2.6)-(2.8) hold.

(9) If $M$ is of type $D_4^{(3)}$, we only need to get rid of the case that short roots $\beta, \gamma\in{\mathcal B}$ such that $(\beta, \gamma)>0$ (respectively, $(\beta, \gamma)<0$) and $\beta+\gamma$ (respectively, $\beta-\gamma$) is still a root. The method is also similar to (4) and (6).
\end{proof}

\begin{lemma}
If $n=2$, then $\lm$ is a generalized Kac-Moody algebra.
\end{lemma}
\begin{proof}
First we may assume that $M$ is a GCM, then $m_{1,2}<0, m_{2,1}<0$. Let $\cl$ be the Kac-Moody algebra with structure matrix $M$,
$\Pi=\{\alpha_1,\alpha_2\}$ be  its prime root system. By Lemma 7, we only need consider $\det M<0$ and we may assume  that $\cl$ is generated by $e_{\alpha_i}, f_{\alpha_i}, h_i (i=1,2)$.

{\noindent\bf Case 1.} If $\Pi'=\{\alpha, \beta\}$ satisfies (2.6)-(2.8), then $\Pi''=\{-\alpha, \beta\}$ satisfies (2.6)-(2.8).

{\noindent\bf Case 2.} If $\Pi'=\{\alpha_1, \rho_{\alpha_1}(\alpha_2)\}$, by case 1, we may assume that $\Pi'=\{-\alpha_1, \rho_{\alpha_1}(\alpha_2)\}$, so $\Pi'$ is still of Kac-Moody type.

Define a map $\varphi$:
\beqs e_{\alpha_2}\mapsto \frac{1}{(-m_{1,2})!}(\ad e_{\alpha_1})^{-m_{1,2}}e_{\alpha_2},&&f_{\alpha_2}\mapsto \frac{1}{(-m_{1,2})!}(\ad (-f_{\alpha_1}))^{-m_{1,2}}f_{\alpha_2},\\
e_{\alpha_1}\mapsto -f_{\alpha_1},&&f_{\alpha_i}\mapsto -e_{\alpha_1},\eeqs
then $\varphi$ determines an isomorphism of $\cl$. Note that a quantum analogue of this isomorphism is the famous Lusztig symmetry(see \cite{L}).
So the Serre relations are preserved. By the definition of braid-equivalent basis, the above two cases are sufficient to show that $\lm=\cl$.  Then $\lm$ is a generalized Kac-Moody algebra.
\end{proof}

{\noindent\bf Examples.}

1) If $$M=\left[\begin{array}{ccc}2&-1&1\\-1&2&-1\\1&-1&2\end{array}\right]$$
and $\Pi=\{\alpha_1,\alpha_2,\alpha_3\}$. In \cite{GX}, it is proved that $\emph{\textbf{gim}}(M)$ has an ideal such that the quotient is isomorphic to $sp_{6}$. In particular, the image of $[e_{\alpha_1},[e_{\alpha_2},e_{\alpha_3}]]$ in the quotient is not zero, then in $\emph{\textbf{gim}}(M)$,
$[e_{\alpha_1},[e_{\alpha_2},e_{\alpha_3}]]\not=0$.  However,
$[e_{\alpha_1},[e_{\alpha_2},e_{\alpha_3}]]=0$ in $\lm$, hence $\lm$ is different from generalized intersection matrix algebra.

Particularly,
$$\Pi\sim\{\alpha_1,\alpha_2,\alpha_3-\alpha_1\},$$
the associated intersection matrix is
$$N=\left(\begin{array}{ccc}2&-1&-1\\-1&2&0\\-1&0&2\end{array}\right),$$
so $\lm\cong\ln$ is a finite dimensional simple Lie algebra of type $A_3$.

2) If $$M=\left(\begin{array}{ccc}2&-1&2\\-1&2&-1\\2&-1&2\end{array}\right)$$
and $\Pi=\{\alpha_1,\alpha_2,\alpha_3\}$, then
$$\Pi\sim\{\alpha_1,\alpha_2,\alpha_3+\alpha_2\}\sim\{\alpha_1,\alpha_2,-\alpha_3-\alpha_2\},$$
the associated intersection matrix is
$$N=\left(\begin{array}{ccc}2&-1&-1\\-1&2&-1\\-1&-1&2\end{array}\right),$$
so $\lm\cong\ln$ is an affine Lie algebra of type $A_2^{(1)}$.

3) If $$M=\left(\begin{array}{cccc}2&2&2&2\\2&2&2&2\\2&2&2&2\\2&2&2&2\end{array}\right)$$
and $\Pi=\{\alpha_1,\alpha_2,\alpha_3,\alpha_4\}$, then $$[e_{\alpha_1},[e_{\alpha_2},[e_{\alpha_3},f_{\alpha_4},]]]\not=0$$ in $\lm$, but
$$[e_{\alpha_1},[e_{\alpha_2},[e_{\alpha_3},f_{\alpha_4},]]]=0$$ in $\emph{\textbf{im}}(M)$, hence $\lm$ is also different from the intersection matrix algebra.

For the relation $[e_{\alpha_1},[e_{\alpha_2},[e_{\alpha_3},f_{\alpha_4},]]]\not=0$ in $\lm$, we shall give a detailed proof in Appendix.

Example 3) says that the length of roots are not limited by the root length of root vector generators, this is very different from Peng-Xu's definition.
As far as the authors know, this phenomenon inherited from ${\bf\emph{gim}}$ disappears in other known quotients. This is another reason driving us to study $\lm$.


\begin{theo}{\label{theo1}}
If $M$ and $N$ are braid-equivalent, then
$$\lm\cong\ln.$$
\end{theo}
\begin{proof}
Without less of generality,  we can suppose that $\lm$ and $\ln$ have the same subspace $H$ and $H^*$ and
$$\Pi_M=\{\alpha_1,\alpha_2\cdots,\alpha_n\},\quad \Pi_N=\{\beta_1,\beta_2,\cdots,\beta_n\},$$
also suppose that

\qquad (1) $\lm$ is generated by $e_{\alpha_i}, f_{\alpha_i}, h_i, d_j$;

\qquad (2) $\ln$ is generated by $x_{\beta_i}, y_{\beta_i}, t_i, s_j$.

{\noindent\bf Case 1.} If $\beta_1=-\alpha_1,\beta_2=\alpha_2,\cdots,\beta_n=\alpha_n$.

It is obvious that the homomorphism $\varphi$ defined via:
\beqs e_{\alpha_1}\mapsto y_{\beta_1},&f_{\alpha_1}\mapsto x_{\beta_1},\\
e_{\alpha_i}\mapsto x_{\beta_i},&f_{\alpha_i}\mapsto x_{\beta_i},&i\not=1,\eeqs
and $d_j\mapsto s_j$ is an isomorphism of Lie algebras.

{\noindent\bf Case 2.} If $\beta_1=\rho_{\alpha_2}(\alpha_1),\beta_2=\alpha_2,\cdots,\beta_n=\alpha_n$. Let $L_M$ be generated by $e_{\alpha_i}, f_{\alpha_i}, h_i (i=1,2)$ and $L_N$ be generated by $x_{\beta_i}, y_{\beta_i}, t_i (i=1,2)$. It suffices to show the Lie algebra isomorphism $L_M\cong L_N$. However, this is a direct consequence of  Lemma 8. Then $\lm\cong\ln$. \end{proof}

Out of question, $\lm$ is a generalized Kac-Moody Lie algebra for $n\leq2$. In the next sections, we always assume that $M$ is indecomposable and $n\geq3$.

\noindent\\[4mm]

\noindent{\bbb 3\quad Classification of $\lm$ for positive and semi-positive definite $M$}\\[0.1cm]

\begin{theo}
Suppose that $M$ is indecomposable. If there exists a diagonal matrix $S$ with positive components, such that $SM$ is positive definite, then $\lm$ is a finite dimensional simple Lie algebra.
\end{theo}
\begin{proof} Since $SM$ is positive definite, then $K=0$.  Suppose that $\Pi=\{\alpha_1,\alpha_2\cdots,\alpha_n\}$ such that
$$\frac{2(\alpha_i,\alpha_j)}{(\alpha_i,\alpha_i)}=m_{ij},$$
and let $\Pi_1=\{\beta_1=\alpha_1\}$, $M_1=(2)$. Clearly, ${\textsf{Pra}}(M_1)$ is a three dimensional simple Lie algebra. Because $M$ is indecomposable, there is $\alpha_i$, we can assume that $i=2$ for convenience, then $\beta_2=\alpha_2$ or $\beta_2=\rho_{\alpha_1}(\alpha_2)$ such that $(\beta_1, \beta_2)<0$. Let $\Pi_2=\{\beta_1,\beta_2\}$ and
$$M_2=\left[\begin{array}{cc}2&\frac{2(\beta_1,\beta_2)}{(\beta_1,\beta_1)}\\
\frac{2(\beta_1,\beta_2)}{(\beta_2,\beta_2)}&2\end{array}\right],$$
then $M_2$ is Cartan matrix and ${\textsf{Pra}}(M_2)$ is a finite dimensional simple Lie algebra. Induction on $p$, let $\Pi_p=\{\beta_1,\cdots,\beta_p\}$ be such that
$$M_p=\left[\frac{2(\beta_i,\beta_j)}{(\beta_j,\beta_j)}\right]_{1\leq i,j\leq p}$$ is a Cartan matrix and $\alpha_{p+1}$ be a non-zero weight of ${\textsf{Pra}}(M_p)$, so there is $\rho$ from ${\textsf{Pra}}(M_p)$'s Weyl group such that $(\beta_i,\beta_{p+1})\leq 0$ for all $1\leq i\leq p$(at least one of them is non-zero), where $\beta_{p+1}=\rho(\alpha_{p+1})$. Note that the existence of $\rho$ follows from the positive definite property of $M_p$.
Let $\Pi_{p+1}=\{\beta_1,\cdots,\beta_{p+1}\}$ and $$M_{p+1}=\left[\frac{2(\beta_i,\beta_j)}{(\beta_j,\beta_j)}\right]_{1\leq i,j\leq p+1}$$ is still a Cartan matrix. Then we can get a Cartan matrix $M_n$ which is braid-equivalent to $M$, so $\lm\cong {\textsf{Pra}}(M_n)$ is a simple Lie algebra of finite type.
\end{proof}

\begin{theo}
Suppose that $M$ is an indecomposable symmetrizable positive or semi-positive definite generalized intersection matrix, then $M$ must be braid-equivalent to an intersection matrix determined by one of the modified Dynkin diagrams
listed in {\bf Figure 1}.
\end{theo}

\centerline{\begin{tabular}{ll}
$A_l(r):$&\xymatrix{\bigcirc\!\!\!\!{r}\;_1\ar@{-}[r]&\circ_2\ar@{-}[r]&\circ\cdots\cdots\circ_{l-1}\ar@{-}[r]&\circ_l}\\
$B_l(r,s):$&\xymatrix{\bigcirc\!\!\!\!{r}\;_1\ar@{-}[r]&\circ_2\ar@{-}[r]&\circ\cdots\cdots\circ\ar@{-}[r]&\circ_{l-1}\ar@{=>}[r]&\bigcirc\!\!\!\!{s}\;_l}\\
$C_l(r,s):$&\xymatrix{\bigcirc\!\!\!\!{r}\;_1\ar@{=>}[r]&\circ_2\ar@{-}[r]&\circ\cdots\cdots\circ\ar@{-}[r]&\circ_{l-1}\ar@{-}[r]&\bigcirc\!\!\!\!{s}\;_l}\\
$D_l(r):$&\xymatrix{\bigcirc\!\!\!\!{r}\;_1\ar@{-}[r]&\circ_2\ar@{-}[r]&\circ\cdots\cdots\circ\ar@{-}[r]&\;\circ_{l-2}\ar@{-}[r]\ar@{-}[d]&\circ_l\\
&&&\;\circ_{l-1}}\\
$E_{6,7,8}(r):$&\xymatrix{\bigcirc\!\!\!\!{r}\;_1\ar@{-}[r]&\circ_3\ar@{-}[r]&\circ_4\ar@{-}[d]\ar@{-}[r]&\circ_{5}\ar@{-}[r]&\circ_6\cdots\cdots\circ\\
&&\circ_{2}}\\
$F_4(r,s):$&\xymatrix{\bigcirc\!\!\!\!{r}\;_1\ar@{-}[r]&\circ_2\ar@{=>}[r]&\circ_3\ar@{-}[r]&\bigcirc\!\!\!\!{s}\;_4}\\
$G_2(r,s):$&\xymatrix{\bigcirc\!\!\!\!{r}\;_1\equiv\!\equiv\!\equiv\!\equiv\!\rangle\bigcirc\!\!\!\!\!{s}\;_2}\\
$A_1(r,s):$&\xymatrix{\bigcirc\!\!\!\!{r}\;_1\underline{\equiv\!\equiv\!\equiv\!\equiv}\!\rangle\bigcirc\!\!\!\!\!{s}\;_2}\\
$BC_l(r,s,t):$&\xymatrix{\bigcirc\!\!\!\!{r}\;_1\ar@{=>}[r]&\bigcirc\!\!\!\!{s}\;_2\ar@{-}[r]&\circ_3\ar@{-}[r]&\circ\cdots\cdots\circ\ar@{-}[r]&\circ_{l}\ar@{=>}[r]&\bigcirc\!\!\!\!{t}\;_{l+1}}
\end{tabular}}

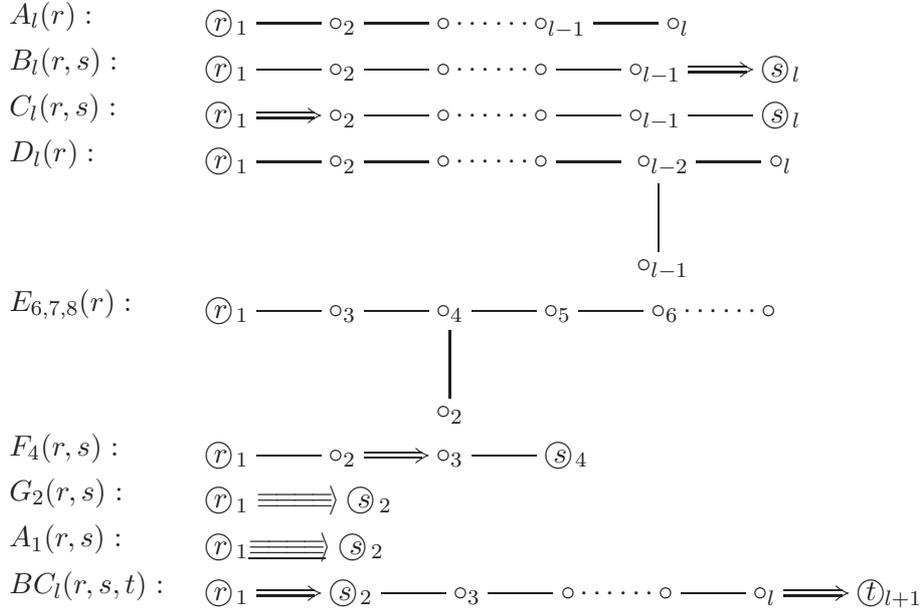
\centerline{\bf Figure 1}

{\bf Interpretation:}

 The circle, the number of lines between circles and the arrows have the same meaning of Dynkin diagrams. The number $r$(or $s,t$) in the circle means the number of copies of the simple root. If any number in circle is $1$, the the diagram is just the Dynkin diagram ($A_1,B_l,C_l,D_l,E_{6,7,8},F_4,G_2$ and $A_{2l}^{(2)}$). For example, the modified Dynkin diagram $B_l(r,s)$ means:

$\bullet$ $\Pi=\{\alpha_{1,i},\alpha_2,\cdots,\alpha_{l-1},\alpha_{l,j}\mid 1\leq i\leq r, 1\leq j\leq s\}$;

$\bullet$ $(\alpha_{1,i},\alpha_{2})=-1$, $(\alpha_p,\alpha_{p+1})=-1$, $(\alpha_{l-1},\alpha_{l,j})=-2$, $(\alpha_{1,i},\alpha_{1,k})=2$,  $(\alpha_p,\alpha_{p})=2$ and $(\alpha_{l,t},\alpha_{l,j})=1$.
Other pairs of roots are orthogonal. Hence $\alpha_{l,j}$ is a short root.

$\bullet$ the intersection matrix is of $(r+s+l-2)\times(r+s+l-2)$.

\begin{proof}
Let $\Pi=\{\alpha_1,\cdots,\alpha_n\}$,  and $\Pi^\sharp=\{\alpha_1,\cdots,\alpha_{n-K}\}$ be such that
$$M_{n-K}=\left[\frac{2(\alpha_i,\alpha_j)}{(\alpha_j,\alpha_j)}\right]_{1\leq i,j\leq n-K}$$
is indecomposable non-degenerate.

Suppose that $W$ is the Weyl group of ${\textsf{Pra}}(M_{n-K})$. By Theorem 10, we may assume that $M_{n-K}$ is a Cartan matrix.  Restricted on dual space of its Cartan subalgebra, for any $\alpha_{j}(j>n-K)$, there exists $w_j\in W$ such that $(w_j(\alpha_j),\alpha_i)\leq0$ for all $1\leq i\leq n-K$.

1) If $\Pi$ has only one root length,  then $\Pi$ is of type  $A_l(r),D_l(r), E_{6,7,8}(r)$.
We prove it in three cases.

(1.a) $\Pi^\sharp$ is of type $E$. Then every $w_j(\alpha_j)$ has to be the minus  highest root(up to an imaginary root), otherwise it contradicts to that $M$ is semi-positive definite and $rank(M)=n-K$. So we may choose $w_j'$ such that $w_j'(\alpha_j)=\alpha_1$ and $\Pi$ is of type  $E_{6,7,8}(r)$.

(1.b) $\Pi^\sharp$ is of type $D$. If there exists $w_j(\alpha_j)$ such that $\Pi^\sharp\cup\{w_j(\alpha_j)\}$ is of type $E_8^{(1)}$, then we may replace $\Pi^\sharp$ by one of type $E_8$ and $\Pi$ is of type $E_8(r)$. For other cases, every $w_j(\alpha_j)$ has to be the minus  highest root So $\Pi$ is of type  $D_l(r)$.

(1.c) $\Pi^\sharp$ is of type $A$. If there exists $w_j(\alpha_j)$ such that $\Pi^\sharp\cup\{w_j(\alpha_j)\}$ is of type $E_7^{(1)}$ or $E_8^{(1)}$, then we may replace $\Pi^\sharp$ by one of type $E_{7,8}$ and $\Pi$ is of type $E_{7,8}(r)$. For other cases, every $w_j(\alpha_j)$ has to be the minus highest root So $\Pi$ is of type  $D_l(r)$.

2) If $\Pi$ has two different root lengths, but $\Pi^\sharp=\{\alpha_1\}$, then $\Pi$ is of type $A_1(r,s)$.

Up to imaginary roots, $\Pi\subset\{\pm\alpha_1, \pm\frac{1}{2}\alpha_1\}$ (or $\Pi\subset\{\pm\alpha_1, \pm2\alpha_1\}$). Equivalently,
$\Pi=\{\alpha_1, -\frac{1}{2}\alpha_1\}$ (or $\Pi=\{\alpha_1, -2\alpha_1\}$), and $\Pi$ is of type $A_1(r,s)$ which depends the number of $\alpha_1$ and the number of $-\frac12\alpha_1$(or $-2\alpha_1$).

3) If $\Pi^\sharp$ is of type $G_2$, then $\Pi$ is of type $G_2(r,s)$.

Each $w_j(\alpha_j)$ has to be the minus highest long root or the minus highest short root, this implies the result.

4) If $\Pi$ has two different root lengths, but $\Pi^\sharp=\{\alpha_1,\alpha_2\}$ has one root length, then $\Pi$ is of type $G_2(r,s)$.

If $w_j(\alpha_j)$ is shorter then the square length of $w_j(\alpha_j)$ has to be $\frac{1}{3}$ of square length of $\alpha_1$. If $w_j(\alpha_j)$ is longer then the square length of $w_j(\alpha_j)$ has to be $3$ multiple of square length of $\alpha_1$. Replace $\Pi^\sharp$ by $\Pi^{*}=\{w_j(\alpha_j), \alpha_2\}$, which is type $G_2$.

5) If $\Pi^\sharp$ is of type $B$ (or $C$) and $\Pi$ has two different root lengths, then $\Pi$ is of type $B_l(r,s)$ (or $C_l(r,s)$). The proof is similar to 3).

In the following cases, the proof is similar and we only state the result.

6) If $\Pi$ has two different root lengths, but $\Pi^\sharp$ is of type $D_l$, then $\Pi$ is of type $B_l(r,s)$ or $C_l(r,s)$.

7) If $\Pi$ has two different root lengths and $\Pi^\sharp$ is of type $B_4$, then $\Pi$ is of type $B_4(r,s)$ or $F_4(r,s)$.

8) If $\Pi$ has two different root lengths and $\Pi^\sharp$ is of type $C_4$, then $\Pi$ is of type $C_4(r,s)$ or $F_4(r,s)$.

9) If $\Pi^\sharp$ is of type $F_4$, then $\Pi$ is of type $F_4(r,s)$.

10) If $\Pi$ has three different root lengths, then $\Pi$ must be of type $BC_l(r,s,t)$.
\end{proof}


\begin{remark}
In \cite{PX}, the same classification to braid-equivalent matrices was given for root systems of GIMs. In this paper, we have provided a different proof.
\end{remark}

\begin{remark}
Let $L$ be a simple Lie algebra of type $X_l$ with a prime root system $\{\alpha_1,\alpha_2,\cdots,\alpha_l\}$, and let $A$ be the Laurent polynomial algebra $\mc[t_1^{\pm1},\cdots,t_\nu^{\pm1}]$ and $\Omega=AdA/dA$. For convenience, we assume that $\alpha_1$ is a long root. As we know that a toroidal Lie algebra of type $X_l$ and nullity $\nu$ can be realized as following:

$$Tor(L)=L\otimes A\oplus\Omega\oplus\bigoplus_{i=1}^\nu\mc d_i.$$

The bracket is given by:
\beqs {[x\otimes t^{\bn},y\otimes t^{\bm}]}&=&[x,y]\otimes t^{\bn+\bm}+\sum_{i=1}^\nu t^{\bn}dt^{\bm},\\
 {[d_j,y\otimes t^{\bm}]}&=&m_jy\otimes t^{\bm},\\
 {[d_j,t^{\bn}dt^{\bm}]}&=&(n_j+m_j)y\otimes t^{\bm},\\
{[t^{\bn}dt^{\bm},y\otimes t^{\bm}]}&=&0,\\
{[t^{\bn}dt^{\bm},t^{\bk}dt^{\br}]}&=&0,\\
{[d_i,d_j]}&=&0,\eeqs
where $t^{\bn}=t_1^{n_1}\cdots t_\nu^{n_\nu}$ and $\sum_{i=1}^\nu n_it^{\bn}t_i^{-1}dt_i.$

Let $e_j=x_{\alpha_j}, f_j=x_{-\alpha_j}, h_j=\alpha_j^\vee$. It is easy to know that $Tor(L)$ can be generated by elements $\{e_1\otimes t_i, f_1\otimes t_i^{-1},e_j,f_j,h_j, d_i|i=1,\cdots,\nu, j=1,\cdots,l\}$

The roots of $\{e_1\otimes t_i,e_j|i=1,\cdots,\nu, j=1,\cdots,l\}$ forms a set $$\{\beta_1=\alpha_1+\delta_1,\cdots,\beta_\nu=\alpha_1+\delta_\nu,\beta_{\nu+1}=\alpha_1,\cdots,\beta_{n+\nu}=\alpha_n\}.$$
Let $$M=\left[\frac{2(\beta_i,\beta_j)}{(\beta_j,\beta_j)}\right]_{(\nu+n)\times(\nu+n)}=\left[\begin{array}{ccccccc}2&\cdots&2&a_{11}&a_{12}&\cdots&a_{1n}\\
\vdots&&\vdots&\vdots&\vdots&&\vdots\\
2&\cdots&2&a_{11}&a_{12}&\cdots&a_{1n}\\
a_{11}&\cdots&a_{11}&a_{11}&a_{12}&\cdots&a_{1n}\\
a_{21}&\cdots&a_{21}&a_{21}&a_{22}&\cdots&a_{2n}\\
\vdots&&\vdots&\vdots&\vdots&&\vdots\\
a_{n1}&\cdots&a_{n1}&a_{n1}&a_{n2}&\cdots&a_{nn}
\end{array}\right].$$
Where $[a_{ij}]_{n\times n}$ is the Cartan matrix of $L$.  It is easy to check that there exists an epimorphism from $\lm$ to $Tor(L)$.
\end{remark}

\noindent\\[4mm]

\noindent{\bbb Appendix: proof of example 3)}\\[0.1cm]

Let $$A=\left(\begin{array}{cccccccc}2&0&0&0&0&-2&-2&-2\\0&2&0&0&-2&0&-2&-2\\0&0&2&0&-2&-2&0&-2\\0&0&0&2&-2&-2&-2&0\\
0&-2&-2&-2&2&0&0&0\\-2&0&-2&-2&0&2&0&0\\-2&-2&0&-2&0&0&2&0\\-2&-2&-2&0&0&0&0&2\end{array}\right)$$
and $\Pi_A=\{\beta_1, \beta_2, \beta_3, \beta_4, \beta_5, \beta_6, \beta_7, \beta_8\}$.

Then Kac-Moody Lie algebra $L_A:=Lie(A)$ has standard generators $x_i, y_i, \kappa_i (1\leq i\leq 8)$, where $\kappa_i$'s span the Cartan subalgebra, $x_i$ has root $\beta_i$ and $y_i$ has root $-\beta_i$. By the definition of Kac-Moody algebra,
\beqs [x_1, [x_2, [x_3, x_8]]]\not=0.\eeqs

Define a degree on generators by:
\beqs \deg(x_i)=-\deg(y_i)=1, \deg(\kappa_i)=0,\eeqs
then the subalgebra $L_A^{\geq0}$ generated by $\kappa_i, x_i(1\leq i\leq 8)$ is a graded Lie algebra. Let $S_0$ be Cartan subalgebra and
\beqs S_1&=&\bigoplus_{i=1}^8\mc x_i,\eeqs
then
\beqs L_A^{\geq0}&=&\bigoplus_{k\geq0}S_k,\eeqs
where $S_k=\sum_{1\leq j\leq k-1}[S_{k-j},S_j]$ for all $k\geq 2$.

Let $g_M:=\emph{\textbf{gim}}(M)$ be the GIM Lie algebra with generators $e_i, f_i, h_i (1\leq i\leq 4)$,
where $e_i=e_{\alpha_i}, f_i=f_{\alpha_i}$ and $h_i=[e_i, f_i]$.

Note that $\lm$ naturally is quotient of $g_M$.

Next we define a degree on $g_M$ by
\beqs \deg(e_i)=\deg(f_i)=1, \deg(h_i)=0.\eeqs
 Also define
\beqs P_0&=&\bigoplus_{i=1}^4\mc h_i,\\
 P_1&=&P_0\oplus\bigoplus_{i=1}^4\mc e_i\oplus \bigoplus_{i=1}^4\mc f_i,\\
 P_k&=&P_{k-1}\oplus \sum_{1\leq j\leq k-1}[P_{k-j},P_j],\;\forall k\geq2,\eeqs
then $\lm$ has a filtration
\beqs P_{-1}:=\{0\}\subset P_0\subset P_1\subset P_2\subset\cdots.\eeqs

Let $G$ be the graded Lie algebra
\beqs G&=&\bigoplus_{k\geq0} G_k,\eeqs
where $G_k=P_k/P_{k-1}$. Let $E_i, F_i$ denote the image of $e_i, f_i$ in $G_1$, respectively.

Define a map  $\phi$ for $1\leq i\leq 4$:
\beqs h_i&\mapsto&\kappa_{i}-\kappa_{i+4},\\
E_i&\mapsto&x_i,\\
F_i&\mapsto&x_{i+4}.\eeqs

We claim that $\phi$ induces a Lie algebra injection from $G$ to $L_A^{\geq0}$.
As a special case, Berman proved that $g_M$ was a fixed point subalgebra of $L_A$(see \cite{Br}). In the proof of isomorphism, he constructed such graded algebras. In his proof, our claim holds,  then we infer that  $[e_1, [e_2, [e_3, f_4]]]\not=0$ in $g_M$.

Now let $I$ be the ideal of $g_M$ such that $\lm=g_M/I$. Note that $$[e_1, [e_2, [e_3, f_4]]]\in P_4,  [e_1, [e_2, [e_3, f_4]]]\not\in P_3.$$

Let $\Gamma=\oplus_{i=1}^4{\mathbb Z}\alpha_i$. For each ${\mathcal B}\sim\Pi$ and each $\alpha\in{\mathcal B}$, we claim that $\alpha\in\alpha_{i_0}+2\Gamma$ for some $i_0\in\{1,2,3,4\}$. This can be checked by the definitions of reflections and matrix $M$. So, if $[e_1, [e_2, [e_3, f_4]]]\in I$,
we have
\beqs \alpha_1+\alpha_2+\alpha_3-\alpha_4=\alpha\pm\beta\in\alpha_{i_0}+\alpha_{j_0}+2\Gamma\eeqs for some $i_0, j_0\in\{1,2,3,4\}$, which is a contradiction.
Then $[e_1, [e_2, [e_3, f_4]]]\not\in I$ and hence $[e_1, [e_2, [e_3, f_4]]]\not=0$ in $\lm$.

\noindent\\[4mm]

\noindent\bf{\footnotesize Acknowledgements}\quad\rm
{\footnotesize The authors are very thankful for the referee's carefully reading, suggestions and comments. The first author is also thankful to Prof. Peng Liangang for providing his  paper joint with Xu. This work is supported by the  NNSF of China (Grant No. 11001110, 11271131), the first author also is supported by  {\it Jiangsu Government Scholarship for Overseas Studies}.
}\\[4mm]

\noindent{\bbb{References}}

\end{document}